\documentclass[draft,onecolumn,12pt]{IEEEtran}

\usepackage{amssymb}
\usepackage{amsmath}
\usepackage{amsthm}

\newtheorem{fed}{Definition}[section]
\newtheorem{teo}[fed]{Theorem}
\newtheorem{cor}[fed]{Corollary}
\newtheorem*{teo*}{Theorem}

\newtheorem{prop}[fed]{Proposition}
\newtheorem{defi}[fed]{Definition}
\theoremstyle{definition}
\newtheorem{rem}[fed]{Remark}

\newtheorem{exa}[fed]{Example}
\newtheorem{exas}[fed]{Examples}

\def\bdem{\begin{proof}}
\def\edem{\renewcommand{\qed}{\hfill $\blacksquare$}
\end{proof}}
\def\AS{[\ese]A}

\def\cW{\mathcal{W}}
\def\cR{\mathcal{R}}

\def\ese{\mathcal{S}}

\def\L{\mathcal{L}}

\def\cF{\mathcal{F}}
\def\cH{\mathcal{H}}
\def\cK{\mathcal{K}}

\def\ese{\mathcal{S}}
\def\ete{\mathcal{T}}
\def\eme{\mathcal{M}}

\newcommand{\pint}[1]{\displaystyle \left \langle #1 \right\rangle}

\begin{document}

\title{Range additivity, shorted operator and the Sherman-Morrison-Woodbury formula}

\author{M. Laura Arias, Gustavo Corach and Alejandra Maestripieri

\thanks{M. Laura Arias is with Instituto Argentino de Matem\'atica ``Alberto P. Calder\'on'',  Buenos Aires, Argentina (e-mail: lauraarias@conicet.gob.ar).

 Gustavo Corach and Alejandra Mestripieri  are with Instituto Argentino de Matem\'atica ``Alberto P. Calder\'on'' and  Dpto. de Matem\'atica, Facultad de Ingenier\'{i}a, Universidad de Buenos Aires, Buenos Aires, Argentina (e-mail: gcorach@fi.uba.ar, amaestri@fi.uba.ar).

Partially supported by UBACYT 20020100100250; 2011-2014}
}

\maketitle

\begin{abstract} We say that two operators $A, B$ have the range additivity property if $R(A+B)=R(A)+R(B).$ In this article we study the relationship between range additivity, shorted operator and certain Hilbert space decomposition known as compatibility. As an application, we extend to infinite dimensional Hilbert space operators a formula by Fill and Fishkind related to the well-known Sherman-Morrison-Woodbury formula. 
\end{abstract}

\section{Introduction}

In this paper we explore some results implied by range additivity of operators in a  Hilbert space $\cH$. Let $L(\cH)$ be the algebra of bounded linear operators on $\cH$ and $L(\cH)^+$ the cone of positive operators on $\cH.$ Consider the set 
$$\cR:=\{(A,B): A,B\in L(\cH) \; {\rm and} \; R(A+B)=R(A)+R(B) \},$$
where $R(T)$ denotes the range of $T$. If $(A,B)\in\cR$ we say that $A,B$ satisfy the {\it range additivity property}. On the other side, we say that a positive operator $A\in L(\cH)^+$ and a closed subspace $\ese\subseteq \cH$ are {\it compatible} if $\ese+(A\ese)^\bot=\cH;$ in \cite{GMS3} it is shown that $A,\ese$ are compatible if and only if there exists an idempotent operator $E\in L(\cH)$ such that $R(E)=\ese$ and $E$ is $A$-selfadjoint, in the sense that $\pint{Ex,y}_A=\pint{x, Ey}_A$ for $x,y\in\cH,$ where $\pint{x,y}_A=\pint{Ax,y}.$ Notice that $||x||_A=\pint{x,x}_A^{1/2}$ is  a seminorm, and that $E$ behaves, with respect to this seminorm, as an orthogonal projection. So, $A$ and $\ese$ are compatible if there is an $A$-orthogonal projection onto $\ese.$ One of the main results of the paper is that $A,\ese$ are compatible if and only if $(A, I-P_\ese)\in\cR$, where $P_\ese$ denotes the classical orthogonal projection onto $\ese.$ Indeed, this is a corollary of the following theorem: for $A,B\in L(\cH)^+$ such that $R(B)$ is closed, then $(A,B)\in \cR$ if and only if $A$ and $N(B)$ are compatible (Theorem \ref{Rcompa}). In order to prove this assertion, and some other general facts on range additivity and compatibility, we explore some features of the shorted operator $\AS.$ This operator has been defined by M. G. Krein \cite{K} as  
$$[\ese]A : =\max \{X \in L(\cH)^+ : X \leq A \;  {\rm and} \;  R(X) \subseteq \ese\}.$$
He proved that the maximum for the L$\rm \ddot{o}$wner ordering (i.e., $C\leq D$ if $\pint{C\xi,\xi}\leq \pint{D\xi,\xi}$ for every $\xi\in\cH$) exists and he applied this construction for a parametrization of the selfadjoint extensions of semi-bounded operators.  W. N. Anderson and G. E. Trapp \cite{AT} redefined and studied this operator, which can be used in the mathematical study of electrical networks. Here, we use the properties of the shorted operator in order to prove that, for $A, B\in L(\cH)^+$ such that $R(B)$ is closed, it holds that $(A,B)\in\cR$ if and only if $A, N(B)$ are compatible where $N(B)$ denotes the nullspace of $B.$ In particular, for $B=I-P_\ese$ we get the assertion above. However, this is not the first manifestation of a relationship between compatibility of $A,\ese$ and properties of $\AS.$ In fact, Anderson and Trapp \cite{AT} prove that $\AS$ is the infimum, for the L$\rm \ddot{o}$wner ordering , of the set $\{EAE^*: E\in L(\cH), E^2=E, N(E)=\ese^\bot\}.$ In \cite[Prop. 4.2]{GMS3}, \cite[Prop. 3.4]{CMS0} it is proven that the infimum is attained if and only if $A,\ese$ are compatible. Moreover, it is proven that if $E\in L(\cH)$ is an idempotent operator such that $AE=E^*A$ and $R(E)=\ese$, then $[\ese^\bot]A=A(I-E).$ Here, we explore more carefully the properties of $\AS$ which are relevant for the compatibility of $A, \ese.$ Another result which may be relevant for updating theory is the extension of the well-known theorem  by J. A. Fill and D. E. Fishkind \cite{FF} which says that, for $n\times n$ complex matrices $A,B$ such that $rk(A+B)=rk(A)+rk(B)$ it holds that  $(A+B)^\dagger=(I-S)A^\dagger (I-T)+SB^\dagger T,$ where $^\dagger$ denotes the Moore-Penrose inverse, 
 $S=(P_{N(B)^\bot}P_{N(A)})^\dagger$ and $T=(P_{N(A^*)}P_{N(B^*)^\bot})^\dagger.$ Here, $rk(X)$ denotes the rank of the matrix $X$ and $P_\eme$ is the orthogonal projection onto the subspace $\eme.$ This is a generalization of a famous formula by J. Sherman, W. J. Morrison and M. A. Woodbury. For a history of this formula see \cite{Ha}. Of course, for Hilbert space operators the rank hypothesis must be replaced by a different one. Since it is well-known that $rk(A+B)=rk(A)+rk(B)$ if and only if $R(A)\cap R(B)=\{0\}$ and $R(A^*)\cap R(B^*)=\{0\},$ we prove that Fill-Fishkind formula holds for $A,B\in L(\cH)$ such that $R(A)$ and $R(B)$ are closed, $R(A)\cap R(B)=R(A^*)\cap R(B^*)=\{0\}$ and $(A,B), (A^*, B^*)\in\cR.$

We end this section introducing some notation. The direct sum between two closed subspaces $\ese$ and $\ete$ will be denoted by $\ese\overset{.}{+}\ete.$ If $\cH=\ese\overset{.}{+}\ete$ then $Q_{\ese//\ete}$ denotes the oblique projection with range $\ese$ and kernel $\ete.$

\section{Range additivity}

Let $\cH,\cK$ be Hilbert spaces. We say that $A,B\in L(\cH,\cK)$ have the {\it range additivity property} if $R(A+B)=R(A)+R(B)$. We denote by $\cR$ the set of all these pairs $(A,B)$, i.e.,
$$\cR:=\{(A,B): A,B\in L(\cH,\cK) \; {\rm and} \; R(A+B)=R(A)+R(B) \}.$$ We collect first some trivial or well-known facts about $\cR.$

\begin{prop}
Let $A,B\in L(\cH,\cK).$ Then
\begin{enumerate}
\item $(A,B)\in\cR$ if and only if $(B,A)\in \cR.$
\item If $R(A)=\cK$ and $A=C+D$ for some $C,D\in L(\cH,\cK)$ then $(C,D)\in \cR.$
\item If $\cH=\cK$ is finite dimensional and $A,B\in L(\cH)^+$ then $(A,B)\in \cR.$
\item If $\cH=\cK,$ $A,B\in L(\cH)^+$ and $R(A+B)$ or $R(A)+R(B)$ is closed, then $(A,B)\in\cR;$ in particular, if $A,B\in L(\cH)^+$, $R(A)$ is closed and $\dim R(B)<\infty$ then $(A,B)\in\cR.$
\end{enumerate}
\end{prop}  
\begin{proof}
Items 1 and 2 are trivial. Item 4 has been proven by Fillmore and Williams \cite[Corollary 3]{FW} under the additional hypothesis that $R(A)$ and $R(B)$ are closed. In \cite[Theorem 3.3]{ACG1} there is a proof without these hypothesis. Items 3 follows from item 4.
\end{proof}

\begin{prop}\label{sad}
For $A,B\in L(\cH,\cK)$ consider the following conditions:
\begin{enumerate}
\item $\overline{R(A^*)}\overset{.}{+}\overline{R(B^*)}$ is closed.
\item $N(A)+N(B)=\cH$
\item $(A,B)\in \cR.$
\end{enumerate}
Then, the next implications hold: $1\Leftrightarrow 2\Rightarrow 3.$ The converse $3\Rightarrow 2$ holds if $R(A)\cap R(B)=\{0\}.$
\end{prop}
\begin{proof}
See \cite[Prop. 5.8]{ACG}. For more general results Corollary \ref{ecua} and Theorem \ref{Alejandra}.
\end{proof}

\begin{exas}\label{ejemplo}
\begin{enumerate}
\item Consider $A=\left(
\begin{array}[pos]{cc}
1 & 1\\
1 & 1	
\end{array}
\right)$ and 
$B=\left(
\begin{array}[pos]{cc}
1 & 0\\
1 & 0	
\end{array}
\right).$ Clearly, $(A,B)\in\cR$ but $(A^*,B^*)\notin \cR.$
\item There exist $C,D\in L(\cH)^+$ such that $R(C),R(D)$ are dense and $(C,D)\notin \cR.$ For this, consider $C,D\in L(\cH)^+$ with dense ranges such that $R(C)\cap R(D)=\{0\}$ (see \cite{FW}). Hence, as $N(C)+N(D)=\{0\}\neq \cH$ then, by Proposition \ref{sad}, $(C,D)\notin \cR.$ 
\end{enumerate}
\end{exas}

We collect now some useful characterizations of $\cR.$ Notice that the proof holds also for vector spaces and modules over a ring.

\begin{prop}\label{sumarangos}
Given $A,B\in L(\cH),$ the following conditions are equivalent:
\begin{enumerate}
\item $(A,B)\in \cR,$
\item $R(A)\subseteq R(A+B)$,
\item $R(B)\subseteq R(A+B),$
\item $R(A-B)\subseteq R(A+B).$
\end{enumerate}
\end{prop}
\begin{proof}
$1\Rightarrow 2, 3.$ If $R(A+B)=R(A)+R(B),$ then,  a fortiori, $R(A)\subseteq R(A+B)$ and $R(B)\subseteq R(A+B).$

$2\Rightarrow 3.$ For every $x\in\cH,$ $Bx=(A+B)x-Ax\in R(A+B).$

$3\Rightarrow 4.$ For every $x\in\cH,$ $(A-B)x=(A+B)x-2Bx\in R(A+B).$

$4\Rightarrow 1.$ For every $x\in\cH,$ $2Ax=(A-B)x+(A+B)x\in R(A+B)$ and $2Bx=-(A-B)x+(A+B)x\in R(A+B),$ and we get $R(A)+R(B)\subseteq R(A+B).$ 

\end{proof}

The next result of R.G. Douglas \cite{Dou} will be frequently used in the paper. 

\begin{teo}
Let $A\in \L(\cH,\cK)$ and $B\in L(\cF, \cK)$. The following conditions are equivalent:
\begin{enumerate}
\item $R(B)\subseteq R(A)$.
\item There is a positive number $\lambda$ such that $BB^*\leq \lambda AA^*$.
\item There exists $C\in L(\cF, \cH)$ such that $AC=B.$
\end{enumerate}
If one of these conditions holds then there is a unique operator $D\in L(\cF, \cH) $ such that $AD=B$  and $R(D)\subseteq N(A)^\bot$. We shall call $D$ the {\textbf{reduced solution}} of $AX=B$.  
\end{teo}

\begin{cor}
For $A,B\in L(\cH)$ the following conditions are equivalent:
\begin{enumerate}
\item the equation $AX=B$ has a solution in $L(\cH).$
\item $(A-B,B)\in\cR.$
\end{enumerate}
\end{cor}

\begin{cor}
For $A,B\in L(\cH)^+$ it holds:
\begin{enumerate}
\item $(A+B)^{1/2}X=A^{1/2}$ has a solution.
\item $(A+B)^{1/2}X=B^{1/2}$ has a solution.
\item $(A^{1/2}, (A+B)^{1/2}-A^{1/2})\in\cR.$
\item $(B^{1/2}, (A+B)^{1/2}-B^{1/2})\in\cR.$
\end{enumerate}
\end{cor}
\begin{proof}
In fact, it holds $A+B\geq A, A+B\geq B$ and Douglas' theorem applies.
\end{proof}

The next corollary complements Proposition \ref{sad}. For a proof see \cite[Prop. 4.13]{ACS}.

\begin{cor}\label{ecua}
For $A,B\in L(\cH,\cK)$ the following conditions are equivalent:
\begin{enumerate}
\item $\overline{R(A^*)}\overset{.}{+}\overline{R(B^*)}$ is closed;
\item equation $(A+B)X=A$ admits a solution which is an oblique (i.e., not necessarily orthogonal) projection in $L(\cH).$ 
\end{enumerate}
\end{cor}

Recall that $A,B\in L(\cH)^+$ are said to be {\it Thompson equivalent} (in symbols, $A\sim_T B$) if there exist positive numbers $r,s$ such that $rA\leq B\leq s A$ (where $C\leq D$ means that $\pint{Cx, x} \leq  \pint{Dx, x}$ for all $x\in\cH$). By Douglas' theorem, $A\sim_T B$ if and only if $R(A^{1/2})=R(B^{1/2}).$ For a fixed $A\in L(\cH)^+$ the Thompson component of $A$ is the convex cone $\{B\in L(\cH)^+: A\sim_T B\}$. The following identity is due to Crimmins (see \cite{FW} for a proof): if $A,B\in L(\cH,\cK)$ then $R(A)+R(B)=R((AA^*+BB^*)^{1/2}).$ Using Crimmins' identity the following result is clear:

\begin{prop}
If $A,B\in L(\cH)^+$ then $(A,B)\in \cR$ if and only if $(A+B)^2\sim_T A^{2}+B^2.$
\end{prop} 

The next characterization of $\cR$ is less elementary than that of Proposition \ref{sad}. Notice, however, that its proof is algebraic, so it also holds in the context of vector spaces, modules over a ring, and so on.

\begin{teo}\label{Alejandra}
Let $A,B\in L(\cH).$ Then $R(A+B)=R(A)+R(B)$ if and only if $R(A)\cap R(B)\subseteq R(A+B)$ and $\cH=A^{-1}({R(B)})+B^{-1}(R(A)).$ In particular, if $R(A)\cap R(B)=\{0\}$ then $(A,B)\in \cR$ if and only if $N(A)+N(B)=\cH.$
\end{teo}
\begin{proof}
Let $T=A+B,$ $\cW=R(A)\cap R(B)$ and suppose that $R(T)=R(A)+R(B).$ Then $R(A)\subseteq R(T)$ and $R(B)\subseteq R(T)$ so that $\cW \subseteq R(T).$ On the other hand, using again that $R(A)$ and $R(B)$ are subsets of $R(T)$ it holds $\cH=T^{-1}(R(T))=T^{-1}(R(A)+R(B))=T^{-1}(R(A))+T^{-1}(R(B)).$ But it is easy to see that $T^{-1}(R(A))=B^{-1}(R(A)).$ Hence, $\cH=T^{-1}(R(A))+T^{-1}(R(B))=A^{-1}({R(B)})+B^{-1}(R(A)).$

Conversely, suppose that $\cW\subseteq R(T)$ and $\cH=B^{-1}(R(A))+A^{-1}(R(B)).$ We shall prove that $R(B)=T(A^{-1}(R(B))).$ In fact, since $B^{-1}(R(A))=B^{-1}(\cW)$ and $A^{-1}(R(B))=A^{-1}(\cW)$ then
$$R(B)=B(\cH)=B(B^{-1}(\cW)+A^{-1}(\cW))=\cW+B(A^{-1}(\cW)),$$
because $\cW\subseteq R(B).$ Moreover, $R(B)=\cW+B(A^{-1}(\cW))=\cW+T(A^{-1}(\cW))=T(A^{-1}(\cW)).$ In fact, for the second equality consider $y\in \cW+B(A^{-1}(\cW))$ then $y=w+Bx$ where $w\in\cW$ and $x\in A^{-1}(\cW),$ so that $y=w-Ax+Tx$ where $w-Ax\in\cW$ and $Tx\in T(A^{-1}(\cW));$ the other inclusion is clear. Then the second equality holds.

To see that $\cW+T(A^{-1}(\cW))=T(A^{-1}(\cW))$ it is sufficient to note that $\cW\subseteq T(A^{-1}(\cW)).$ In fact, $T^{-1}(\cW)=A^{-1}(\cW)\cap B^{-1}(\cW)\subseteq A^{-1}(\cW) $ then applying $T$ to both sides of the inclusion $\cW=TT^{-1}(\cW)\subseteq T(A^{-1}(\cW))$ because $\cW\subseteq R(T).$

Hence, $R(B)=T(A^{-1}(\cW))=T(A^{-1}(R(B))\subseteq R(T).$ Applying Proposition \ref{sumarangos}, $(A,B)\in \cR.$

\end{proof}

One of the obstructions for range additivity for operators in Hilbert spaces is that $R(A)$ is, in general, non closed. Therefore, the identity $R(A+B)=R(A)+(B)$ is not equivalent to $N(A^*+B^*)=N(A^*)\cap N(B^*),$ which is easier to check. On these matters, see the papers by P. $\check{\rm S}$emrl \cite[\S 2]{Se} and G. L$\check{\rm e}$snjak and P. $\check{\rm S}$emrl \cite{LS}, where they discuss different kinds of topological range additivity properties. See also the paper by J. Baksalary, P. $\check{\rm S}$emrl and G. P. H. Styan \cite{BSS}.

\section{Shorted operators and range additivity}\label{secshorted}

In his paper on selfadjoint extensions of certain unbounded operators \cite{K}, M. G. Krein defined for the first time a shorted operator (this is modern terminology). More precisely, if $A\in L(\cH)^+$ and $\ese$ is a closed subspace of $\cH,$ Krein proved that the set
$$\{C \in L(\cH)^+ : C \leq A \;  {\rm and} \;  R(C) \subseteq \ese\}$$
admits a maximal element $[\ese]A.$ Moreover, Krein proved that 
$$[\ese]A=A^{1/2}P_\eme A^{1/2},$$ 
if $\eme=A^{-1/2}(\ese).$  Krein constructed the shorted operators to find selfadjoint positive extensions of certain unbounded operators. For a modern exposition of Krein's ideas on these matters, see \cite{Ar}. 

Later, W. N. Anderson and G. E. Trapp \cite{AT} rediscovered the operator $[\ese]A,$ proved many useful properties and showed its relevance in the theory of impedance matrices of networks. The papers by E. L. Pekarev \cite{Pe}, Pekarev and Smul'jan \cite{PS}, T. Ando \cite{A}  and S. L. Eriksson and H. Leutwiler \cite{EL} contain many useful theorems about Krein shorted operators. A nice exposition for shorted operators in finite dimensional spaces is that of T. Ando \cite{A}.  It is worth mentioning that there is a binary operation between positive operators, the {\it parallel sum}, which is also relevant in electrical network theory and which is related to shorted operators. If $A,B$ are the impedance matrices of two $n$-port resistive networks then $A:B:=A(A+B)^\dagger B$ is the impedance matrix of their parallel connection. For positive operators $A,B$ on a Hilbert space $\cH$, Fillmore and Williams \cite{FW} defined
$$A:B=A^{1/2}C^*DB^{1/2},$$
if $C$ (resp. $D$) is the reduced solution of $(A+B)^{1/2}X=A^{1/2}$ (resp. $(A+B)^{1/2}X=B^{1/2}$). 

Anderson and Trapp \cite{AT} proved that $A:B$ is the $(1,1)$ entry of $[\ese] \left(
\begin{array}[pos]{cc}
A & A\\
A & A+B	
\end{array}
\right),$ if $\ese=\cH\oplus \{0\} $ and the matrix $\left(
\begin{array}[pos]{cc}
A & A\\
A	 & A+B	
\end{array}
\right)$ is considered as an element of $L(\cH\oplus \cH)^+.$ Thus, the parallel addition is a particular form of the shorted operation. 
Any extension to non necessarily positive operators of the parallel sum operation requires that $(A, B)$ and $(A^*, B^*)$ belong to $\cR,$ at least if one wants to keep the desirable commutativity $A:B=B:A$ \cite[10.1.6]{Rao}.  Indeed, Rao and Mitra say that $A,B$ are {\it parallel summable} if $A(A+B)^-B$ is invariant for any generalized inverse of $A+B.$ It turns out that this happens if and only if $(A,B)\in\cR$ and $(A^*, B^*)\in\cR.$ This means that there is an strong relationship among Krein shorted operators, Douglas range inclusion and range additivity.

We collect in the next proposition some facts on the Krein shorted operators, mainly extracted from the paper \cite{AT} by Anderson and Trapp. 

A warning about notation. The original notation by Krein is $A_\ese$. Anderson and Trapp \cite{AT} used $\ese(A).$ Ando \cite{A3} proposed $[\ese]A.$ This is coherent with a relevant construction $[B]A$ for $A,B\in L(\cH)^+$ that he defined and studied in \cite{A1}, by generalizing a theorem of Anderson and Trapp that  $([\ese]A)x=\lim_{n\rightarrow\infty}(A:n P_\ese)x$ for every $x\in\cH.$ Ando defined the existence of  $([B]A)x=\lim_{n\rightarrow\infty}(A:n B)x$ for every $x\in\cH$ and proved many relevant results on this construction. In particular, it holds that $[\ese]A=[B]A$ if $\ese=R(B).$ Erikson and Leutwiler \cite{EL} used $Q_BA$ for Ando's $[B]A.$ In \cite{A}, Ando has used $A_{/\ese}$ for the shorted operator and $A_\ese=A-A_{/\ese}.$ Corach, Maestripieri and Stojanoff used $\sum(P_\ese,A)$ in \cite{ GMS3} and $A_{/\ese}$ in \cite{CMS0} to denote what we are denoting now $[\ese^\bot]A.$

\begin{prop}\label{propiedadesShort}
Given $A,B\in L(\cH)^+$ and closed subspaces $\ese,\ete$ of $\cH$ the following properties hold:
\begin{enumerate}
\item $R(A)\cap\ese\subseteq R([\ese]A)\subseteq R(([\ese]A)^{1/2})=R(A^{1/2})\cap\ese;$ in particular, $R([\ese]A)$ is closed if $R(A)$ is closed or, more generally, if $R(A)\cap\ese=R(A^{1/2})\cap\ese.$
\item $N([\ese]A)=N(P_{A^{-1/2}\ese}A^{1/2})=A^{-1/2}\overline{A^{1/2}(\ese^\bot)}\supseteq N(A)+\ese^\bot;$ equality holds if and only if $\overline{A^{1/2}(\ese^\bot})\cap R(A^{1/2})=A^{1/2}(\ese^\bot)$.
\item $[\ese](A+B)\geq [\ese]A+[\ese]B;$ equality holds if and only if $R((A-[\ese]A+B-[\ese]B)^{1/2})\cap\ese=\{0\}.$
\item $R((A-[\ese]A)^{1/2})\cap\ese=\{0\}.$ In particular, $\overline{R(\AS)}\cap R(A-\AS)=\{0\}.$
  
\end{enumerate} 
\end{prop}

\begin{proof}
\begin{enumerate}
\item See \cite[Corollary 4 of Theorem 1 and Corollary of Theorem 3]{AT}
\item See \cite[Corollary 2.3]{CMS0}
\item See \cite[Theorem 4]{AT}.
\item See \cite[Theorem 2]{AT}.

\end{enumerate}
\end{proof}

\begin{cor}\label{AS=S}
Let $A,B\in L(\cH)^+.$ Then:
\begin{enumerate} 
\item If $\ese=\overline{R(B)}$ then $[\ese]B=B$ and $[\ese](A+B)=[\ese]A+B.$
\item If $\ese=R(B)$ is closed then $R([\ese](A+B))=\ese$ and $N([\ese](A+B))=\ese^\bot.$
\end{enumerate}
\end{cor}
\begin{proof}
\begin{enumerate}
\item The identity $[\ese]B=B$ can be checked through the definition of $[\ese]B;$ the identity $[\ese](A+B)=[\ese]A+B$ follows from items 3 and 4 in Proposition \ref{propiedadesShort}.
\item For every $C\in L(\cH)^+$ it holds $R(([\ese]C)^{1/2})\subseteq \ese,$ therefore $\ese\supseteq R(([\ese](A+B))^{1/2})= R(([\ese]A+B)^{1/2})=R(([\ese]A)^{1/2})+\ese \supseteq \ese,$ where the second equality holds by Crimmins' identity.The kernel condition follows by taking orthogonal complement.
\end{enumerate}
\end{proof}

\begin{prop}\label{rangoAS} Let $A\in L(\cH)^+$ and let $\ese$ be a closed subspace of $\cH$. 
The following conditions are equivalent:
\begin{enumerate}
\item $(\AS, A-\AS)\in\cR;$
\item $R(A)=R(A-\AS)\overset{.}{+}R(\AS);$
\item $R(\AS)\subseteq R(A);$
\item $R(A^{1/2})=\eme\cap R(A^{1/2})\oplus\eme^\bot\cap R(A^{1/2}),$ if $\eme=A^{-1/2}(\ese).$
\end{enumerate}
\end{prop}
\begin{proof}
Notice that $N(\AS)=A^{-1/2}(\overline{A^{1/2}(\ese^\bot)})$ and $N(A-\AS)=A^{-1}(\ese).$

$1\Leftrightarrow 2\Leftrightarrow 3.$ It follows by Proposition \ref{sumarangos} and Proposition \ref{propiedadesShort}.

$3\Leftrightarrow 4.$ Assume that $R(\AS)\subseteq R(A)$ and let $y=A^{1/2}x\in R(A^{1/2}).$ Hence, $A^{1/2}x=P_{\eme}A^{1/2}x+(I-P_\eme)A^{1/2}x.$ Applying $A^{1/2}$ in both sides, we get that $Ax=A^{1/2}(I-P_\eme) A^{1/2}x+\AS x.$ Thus, since $R(\AS)\subseteq R(A)$ we obtain that $A^{1/2}(I-P_\eme) A^{1/2}x\in R(A).$ Therefore, $A^{1/2}(I-P_\eme) A^{1/2}x=Az$ for some $z\in\cH.$ From this, $(I-P_\eme) A^{1/2}x-A^{1/2}z\in N(A)\cap \overline{R(A)}=\{0\},$ i.e., $(I-P_\eme) A^{1/2}x=A^{1/2}z\in R(A^{1/2})\cap\eme^\bot.$ Therefore, $A^{1/2}x=P_{\eme}A^{1/2}x+(I-P_\eme)A^{1/2}x\in \eme\cap R(A^{1/2})\oplus\eme^\bot\cap R(A^{1/2})$ and item 3 is proved.

Conversely, assume that  $R(A^{1/2})=\eme\cap R(A^{1/2})\oplus\eme^\bot\cap R(A^{1/2}).$ Hence, $R(\AS)=R(A^{1/2}P_\eme A^{1/2})\subseteq A^{1/2}(\eme\cap R(A^{1/2}))\subseteq R(A).$

\end{proof}

\section{Compatibility and range additivity} \label{secaditividad} 

\begin{defi}\label{defcompa}
 Given $A\in L(\cH)^+$ and $\ese$ a closed subspace of $\cH$, we say that the pair $A,\ese$ is {\bf{compatible}} if $\cH=\ese+(A\ese)^\bot.$
\end{defi}

As shown in \cite{GMS3} the compatibility of a pair $A,\ese$ means that there exists a (bounded linear) projection with image $\ese$ which is Hermitian with respect to the semi-inner product $\pint{ \cdot, \cdot}_A$ defined by $\pint{\xi, \eta}_A=\pint{A\xi,\eta}.$ It is worth mentioning that compatibility gives a kind of weak version of invariant subspaces. In fact, if $A$ is a selfadjoint operator on $\cH$ and $\ese$ is a closed subspace, then $\ese$ is an invariant subspace for $A$ if $A\ese\subseteq \ese,$ which means that $P_\ese A P_\ese=P_\ese A.$ On the other side, $A,\ese$ are compatible if and only if $R(P_\ese A P_\ese)=R(P_\ese A);$ for a proof of this fact see \cite[Proposition 3.3]{GMS3}. In the recent paper \cite[Proposition 2.9]{ACG} it is proven that $A, \ese$ are compatible if and only if $(P_\ese A, I-P_\ese)\in\cR.$ In this section we shall complete this result by proving that $A,\ese$ are compatible if and only if $(A, I-P_\ese)\in\cR.$

\begin{prop}\label{compa}\cite[Theorem 3.8]{CMS0}
Let $A\in L(\cH)^+$ and $\ese$ a closed subspace of $\cH.$ The following conditions are equivalent:
\begin{enumerate}
\item $(A,\ese)$ is compatible.

\item $R([\ese^\bot]A)\subseteq R(A)$ and $N([\ese^\bot]A)=N(A)+\ese.$
\end{enumerate}
\end{prop}

\begin{prop}
Let $A,B\in L(\cH)^+$ with closed ranges. The next conditions are equivalent:
\begin{enumerate}
\item $A, N(B)$ are compatible.
\item $N(A)+N(B)$ is closed.
\item $B, N(A)$ are compatible.
\item $R(A)+R(B)$ is closed.
\item $(A,B)\in\cR.$ 
\end{enumerate}
\end{prop}
\begin{proof}
$1\Leftrightarrow 2.$ \cite[Theorem 6.2]{GMS3}.

$2\Leftrightarrow 3.$ Idem.

$2\Leftrightarrow 4.$ It follows from the general fact that, for closed subspaces $\ese,\ete$ then $\ese+\ete$ is closed if and only if $\ese^\bot+\ete^\bot$ is closed. See \cite[Theorem 13]{Deu}.

$4\Rightarrow 5.$ See \cite[Corollary 3]{FW}.

$5\Rightarrow 4.$ $R(A+B)=R(A)+R(B)=R(A^{1/2})+R(B^{1/2})=R((A+B)^{1/2})$ by Crimmins' identity. Then $R(A+B)$ is closed and so $R(A)+R(B)$ is closed .
\end{proof}

\begin{teo}\label{Rcompa}
Let $A,B\in L(\cH)^+$ and suppose that $B$ has a closed range. The following conditions are equivalent:
\begin{enumerate}
\item $A, N(B)$ are compatible.
\item $(A,B)\in\cR.$
\item $R(B)\overset{.}{+}\overline{AN(B)}$ is closed.
\end{enumerate}
\end{teo}
\begin{proof}
$1\Leftrightarrow 2.$
Let $\ese=N(B).$ First observe that $A,\ese$ are compatible if and only if $A+B,\ese$ are compatible. Indeed, $\ese+((A+B)\ese)^\bot=\ese+(A\ese)^\bot.$ Hence, by Proposition \ref{compa}, $A,\ese$ are compatible if and only if $R([\ese^\bot](A+B))\subseteq R(A+B)$ and $N([\ese^\bot](A+B))=\ese+N(A+B)$ or, equivalently, by Corollary \ref{AS=S}, $\ese^\bot\subseteq R(A+B)$ (notice that $N(A+B)=N(A)\cap N(B)\subseteq \ese)$.  Summarizing, $A,\ese$ are compatible if and only if $R(B)=\ese^\bot\subseteq R(A+B),$ i.e., $R(A+B)=R(A)+R(B).$

$1\Leftrightarrow 3.$ It follows applying  \cite[Theorem 13]{Deu}. 
\end{proof}

\begin{cor}
Let $A\in L(\cH)^+$ and $\ese$ a closed subspace of $\cH.$ The next conditions are equivalent:
\begin{enumerate}
\item $A,\ese$ are compatible;
\item $(A,I-P_\ese)\in\cR.$
\item $\ese^\bot\overset{.}{+}\overline{A \ese}$ is closed.
\end{enumerate}	
\end{cor}

\begin{prop}\label{rangosdisjuntos}
Let $A,B\in L(\cH)^+$ such that $R(A)\cap\overline{R(B)}=\{0\}.$ Then, $(A,B)\in\cR$ if and only if $A, N(B)$ are compatible.
\end{prop}
\begin{proof}
Since, $R(A)\cap R(B)=\{0\}$ then $R(A+B)=R(A)+R(B)$ if and only if $\cH=N(A)+N(B).$ Now, $N(A)+N(B)=A^{-1}(\{0\})+N(B)=A^{-1}(\overline{R(B)})+N(B)=A^{-1}(N(B)^\bot)+N(B).$ Therefore, $R(A+B)=R(A)+R(B)$ if and only if $\cH=A^{-1}(N(B)^\bot)+N(B),$ i.e., if and only if $A, N(B)$ are compatible.
\end{proof}

The next example shows that the compatibility of the pair $A, N(B)$ does not imply, in general, that $(A,B)\in\cR.$
\begin{exa}

Considering $C$ and $D$ as in Example \ref{ejemplo}.2, we define
$A=\left(
\begin{array}[pos]{cc}
0 & 0\\
0	 & C	
\end{array}
\right)$ and 
$B=\left(
\begin{array}[pos]{cc}
0 & 0\\
0	 & D	
\end{array}
\right).$
Clearly, $(A,B)\notin \cR.$ However, $A, N(B)$ are compatible. 
\end{exa}

\begin{cor}\label{RA} Let $A\in L(\cH)^+$ and $\ese$ a closed subspace of $\cH$.
The following conditions are equivalent:
\begin{enumerate}
\item $R(A)=R(A-\AS)\overset{.}{+}R(\AS);$
\item $A-\AS, N(\AS)$ are compatible;
\item $A, N(\AS)$ are compatible.

\end{enumerate}
\end{cor}
\begin{proof}
$1\Leftrightarrow 2.$ It follows from Proposition \ref{propiedadesShort} and Proposition \ref{rangosdisjuntos}.  

$2\Leftrightarrow 3.$ It follows from the fact that $A=\AS+A-\AS.$
\end{proof}

\section{The Fill-Fishkind formula}

This last section is devoted to the Fill-Fishkind formula. In order to identify certain Moore-Penrose inverses of products of orthogonal projections, the next theorem (due to Penrose and Greville) will be helpful.

\begin{teo}\label{PQ+}
If $Q\in L(\cH)$ is an oblique projection then $Q^\dagger = P_{N(Q)^\bot}P_{R(Q)}$. Conversely, if $M$ and $N$ are closed subspaces of $\cH$ such that $P_MP_N$ has closed range,
then $(P_MP_N)^\dagger$ is the unique oblique projection with range $R(P_NP_M)$ and nullspace
$N(P_NP_M)$.
\end{teo}
\begin{proof}
For matrices, the proof appears in the paper by Penrose \cite[Lemma 2.3]{P} and Greville \cite[Theorem 1]{G}. For general Hilbert spaces, see \cite[Theorem 4.1]{CM}.
\end{proof}

We prove now the extension of the theorem by Fill and Fishkind \cite[Theorem 3]{FF} mentioned in the introduction.

\begin{teo}
Let $A,B\in L(\cH,\cK)$ such that $R(A), R(B)$ are closed, $R(A)\cap R(B)=R(A^*)\cap R(B^*)=\{0\}$ and $(A,B)\in\cR$ and $(A^*,B^*)\in\cR$. Hence,
\begin{equation}\label{formula}
(A+B)^\dagger=(I-S)A^\dagger (I-T)+SB^\dagger T,
\end{equation}
where $$S=(P_{N(B)^\bot}P_{N(A)})^\dagger=Q_{P_{N(A)}(N(B)^\bot)//N(B)}$$ 
and $$T=(P_{N(A^*)}P_{N(B^*)^\bot})^\dagger=Q_{R(B)//R(A)+R(A)^\bot\cap R(B)^\bot}.$$
\end{teo}
\begin{proof} We show first that all Moore-Penrose inverses which appear in (\ref{formula}) are bounded. In fact, by Proposition \ref{sad}, $R(A)\overset{.}{+} R(B)$ and  $R(A^*)\overset{.}{+} R(B^*)$ are closed and so $R(A+B)$ is also closed. Therefore, in addition,  $ P_{N(B)^\bot}P_{N(A)} $  and $P_{N(A^*)}P_{N(B^*)^\bot}$ have closed ranges because of \cite[Theo. 22]{Deu}.
In order to prove that $X$ is the Moore-Penrose inverse of $A$ is suffices to prove that $ AX=P_{R(A)},$ $ XA=P_{R(A^*)}$ and $ XAX=X$.   In our case, we shall prove: 
\begin{enumerate}
\item[i)] $(A+B)((I-S)A^\dagger (I-T)+SB^\dagger T)=P_{R(A+B)}$ 
\item[ii)] $((I-S)A^\dagger (I-T)+SB^\dagger T) (A+B)=P_{R(A^*+B^*)}.$ 
\item[iii)] $((I-S)A^\dagger (I-T)+SB^\dagger T) (A+B)((I-S)A^\dagger (I-T)+SB^\dagger T)=(I-S)A^\dagger (I-T)+SB^\dagger T.$ 
\end{enumerate}
By Theorem \ref{PQ+}, we have that 
$S=Q_{P_{N(A)}(N(B)^\bot)//N(B)}$ and $T=Q_{R(B)//R(A)+R(A)^\bot\cap R(B)^\bot}.$ Therefore,

\begin{enumerate}
\item[i)] After computations, we obtain that: $(A+B)((I-S)A^\dagger (I-T)+SB^\dagger T)=Q_1+T$ where $Q_1=Q_{R(A)//R(B)+R(A)^\bot\cap R(B)^\bot}.$ 
Therefore:
\begin{enumerate}
\item Since $Q_1 T=TQ_1=0$ then $Q_1+T$ is a projection.
\item  Clearly, $R(Q_1+T)\subseteq R(A+B).$ On the other side, as $(Q_1+T)(A+B)=A+B$ we get the other inclusion, and so $R(Q_1+T)=R(A+B).$
\item Finally, as $R(A+B)^\bot=R(A)^\bot\cap R(B)^\bot\subseteq N(Q_1+T)$ we obtain that $Q_1+T=P_{R(A+B)}$ as desired.
\end{enumerate}  
\item[ii)] After computations, we obtain that: $((I-S)A^\dagger (I-T)+SB^\dagger T)(A+B)=I-(I-S)P_{N(A)}$.
\begin{enumerate}
\item Notice that $(I-P_{N(A)})(I-S)=I-P_{N(A)},$ then $(I-S)P_{N(A)}=P_{N(A)}(I-S)P_{N(A)}=(I-S)P_{N(A)}(I-S)$ and so $(I-(I-S)P_{N(A)})^2=I-(I-S)P_{N(A)}.$
\item  Clearly, $N(A)\cap N(B)\subseteq N(A+B)\subseteq N(I-(I-S)P_{N(A)}).$ For the other inclusion, if $x\in N(I-(I-S)P_{N(A)})$ and since $(I-S)P_{N(A)}=(I-S)P_{N(A)}(I-S)$ we have that $x=(I-S)P_{N(A)}x\in N(B)$ and $x=P_{N(A)}(I-S)P_{N(A)}x\in N(A),$ i.e., $x\in N(A)\cap N(B).$ Therefore,  $ N(I-(I-S)P_{N(A)})=N(A)\cap N(B)=(R(A^*)+R(B^*))^\bot=R(A^*+B^*)^\bot.$ 
\item Finally, as $I-(I-S)P_{N(A)}(A^*+B^*)=A^*+B^*$ we get that $R(A^*+B^*)\subseteq R(I-(I-S)P_{N(A)})$ and so, by the previous items, we conclude that $((I-S)A^\dagger (I-T)+SB^\dagger T)(A+B)=I-(I-S)P_{N(A)}=P_{R(A^*+B^*)}$ as desired. 
\end{enumerate}
\item[iii)] As $((I-S)A^\dagger (I-T)+SB^\dagger T)(A+B)=I-EP_{N(A)}$ where $E=Q_{N(B)//P_{N(A)}(N(B)^\bot)}$. Then, $((I-S)A^\dagger (I-T)+SB^\dagger T) (A+B)((I-S)A^\dagger (I-T)+SB^\dagger T)=(I-EP_{N(A)})((I-S)A^\dagger (I-T)+SB^\dagger T)=(I-S)A^\dagger (I-T)+SB^\dagger T$ because $EP_{N(A)}((I-S)A^\dagger (I-T)+SB^\dagger T)=0$ since $EP_{N(A)}S=0=EP_{N(A)}A^\dagger.$
\end{enumerate}
\end{proof}

\begin{rem}
Fill and Fishkind proved their formula under the hypothesis $rk(A+B)=rk(A)+rk(B)$ where $A,B$ are $n\times n-$complex matrices and $rk$ denotes the rank. It is well known that this rank additivity is equivalent to $R(A)\cap R(B)=R(A^*)\cap R(B^*)=\{0\}.$ Moreover,  by Proposition \ref{sad},  $R(A)\cap R(B)=R(A^*)\cap R(B^*)=\{0\}$ is equivalent (for matrices) to $(A,B), (A^*,B^*)\in\cR.$ Thus, there is no loss in this generalization. For a quite different set of hypothesis for Fill-Fishkind formula in Hilbert spaces, see the paper by Deng \cite{Deng}.
\end{rem}

\end{document}